\documentclass[11pt,a4paper]{article}
\pdfoutput=1

\usepackage[utf8]{inputenc}
\usepackage[T1]{fontenc}

\usepackage{amsmath}
\usepackage{amsfonts}
\usepackage{amssymb}
\usepackage{amsthm}
\usepackage{thmtools}
\usepackage{enumerate}
\usepackage{mathtools} 
\usepackage{xfrac}
\usepackage{lmodern}
\usepackage[nottoc]{tocbibind} 
\usepackage{hyperref} 
\usepackage[capitalise]{cleveref} 
\usepackage{graphicx}
\usepackage[english]{babel}
\usepackage{slashed}
\usepackage{microtype}
\usepackage[autostyle=true]{csquotes}
\usepackage[all,cmtip]{xy}
\usepackage{tikz-cd}
\usepackage{verbatim}
\usepackage[normalem]{ulem}

\newcounter{comments}
\newenvironment{displaycomment}{\begin{list}{}{\rightmargin=1cm\leftmargin=1cm}\item\sf\begin{small}}{\end{small}\end{list}}
\def\comments{
  \setcounter{comments}{1}
  
  }
\comments

\usepackage[a4paper,
            top=2.5cm,
            bottom=2.5cm,
            inner=3.0cm,
            outer=3.0cm]{geometry}

\usepackage[citestyle=alphabetic,
            bibstyle=alphabetic,
            backend=bibtex,
            url=true,
            doi=true,
            isbn=false,
            giveninits=true]{biblatex}

\bibliography{Bibliography_Chern_loc_trace_transgress}
\AtEveryBibitem{%
  \ifentrytype{article}{
    \clearfield{url}%
    \clearfield{urldate}%
  }{}
  \ifentrytype{book}{
    \clearfield{url}%
    \clearfield{urldate}%
  }{}
  \ifentrytype{incollection}{
    \clearfield{url}%
    \clearfield{urldate}%
  }{}
}
\usepackage{xparse} 
\usepackage{enumitem}
\setlist[enumerate]{label={\upshape(\roman*)}}
\newlist{myenumi}{enumerate}{1}
\setlist[myenumi,1]{label=\upshape(\roman*)}
\newlist{myenuma}{enumerate}{1}
\setlist[myenuma,1]{label=\upshape(\alph*)}

\declaretheorem[name=Theorem,numberwithin=section]{thm}
\declaretheorem[name=Theorem A,numbered=no]{thmA}
\declaretheorem[name=Theorem,numbered=no]{thmnonumber}
\declaretheorem[name=Theorem B,numbered=no]{thmB}
\declaretheorem[name=Theorem C,numbered=no]{thmC}
\declaretheorem[name=Lemma,numberlike=thm]{lem}
\declaretheorem[name=Corollary,numberlike=thm]{cor}
\declaretheorem[name=Proposition,numberlike=thm]{prop}

\declaretheorem[name=Definition,numberlike=thm,style=definition,qed=\(\blacklozenge\)]{defn}
\declaretheorem[name=Remark,numberlike=thm,style=definition,qed=\(\blacklozenge\)]{rem}

\crefdefaultlabelformat{#2\textup{#1}#3}
\crefname{thm}{Theorem}{Theorems}
\crefname{lem}{Lemma}{Lemmas}
\crefname{defn}{Definition}{Definitions}
\crefname{prop}{Proposition}{Propositions}
\crefname{cor}{Corollary}{Corollaries}
\crefname{equation}{}{}

\newcommand{\IN}{\mathbb{N}}
\newcommand{\Z}{\mathbb{Z}}

\newcommand{\R}{\mathbb{R}}
\newcommand{\IR}{\mathbb{R}}
\newcommand{\C}{\mathbb{C}}

\newcommand{\toprm}{\mathrm{top}}

\newcommand{\lf}{\mathrm{lf}}
\newcommand{\an}{\mathrm{an}}
\newcommand{\ch}{\operatorname{ch}}

\newcommand{\alg}{\mathrm{alg}}

\newcommand{\Tr}{\mathrm{Tr}}

\newcommand{\fin}{\mathrm{fin}}

\newcommand{\dR}{\mathrm{dR}}

\newcommand{\cP}{\mathcal{P}}

\newcommand{\cA}{\mathcal{A}}

\newcommand{\cB}{\mathcal{B}}

\newcommand{\cC}{\mathcal{C}}

\newcommand{\HX}{H \! X}

\newcommand{\KH}{K \! H}
\newcommand{\HC}{H \! C}
\newcommand{\HP}{H \! P}

\DeclareMathOperator{\Hom}{Hom}

\newcommand{\per}{\mathrm{per}}

\counterwithin{equation}{section} 

\title{Transgressing the algebraic coarse character map}

\date{}

\author{Alexander Engel\thanks{
Universit\"at Greifswald,
Walther--Rathenau--Stra{\ss}e 47,
17489 Greifswald,
Germany\newline
\href{mailto:alexander.engel@uni-greifswald.de}{alexander.engel@uni-greifswald.de}
}
\and
Matthias Ludewig\thanks{
Universit\"at Greifswald,
Walther--Rathenau--Stra{\ss}e 47,
17489 Greifswald,
Germany\newline
\href{mailto:matthias.ludewig@uni-greifswald.de}{matthias.ludewig@uni-greifswald.de}
}
}

\begin{document}

%
%

\maketitle

\begin{abstract}
We pursue an old conjecture of John Roe about the algebraic $K$-theory of the algebra of finite propagation, locally trace-class operators, namely that transgressing the algebraic coarse character map on this algebra to a Higson dominated corona coincides with the usual Chern character on the corona.
\end{abstract}

\setcounter{tocdepth}{2}
\tableofcontents

\section{Introduction}

To any proper metric space $M$, one has a canonically associated ``boundary at infinity'', its Higson corona $\partial_h M$ (see Definition \ref{defn_higson_corona}).
There exists an ``index pairing'', which pairs classes in the $K$-theory group $K^{p-1}(\partial_h M)$ with classes in the $K$-theory group $K_p^{\mathrm{top}}(C^*M)$ of the Roe algebra of $M$, to obtain an integer \cite{EmeMey}. This is the pairing on the right hand side of \eqref{PairingEquation} below.

On the other hand, given a class $x$ in $K^{p-1}(\partial_h M)$, we may form the Chern character $\ch(x)$ in Alexander--Spanier cohomology of $\partial_h M$ \cite{gorokhovsky} and then transgress to obtain a coarse cohomology class $T\ch(x)$ in $\HX^p(M)$.
Coarse cohomology does not directly pair with $K$-theory classes in the Roe algebra, but it does pair (via the Chern character in cyclic homology) with algebraic $K$-theory classes over the dense subalgebra $\mathcal{B}M \subset C^*M$, the algebra of locally trace-class, finite propagation operators on $M$ \cite{roe_coarse_cohomology}. This is the pairing on the left hand side of \eqref{PairingEquation}.

In total, we have described two different ways to obtain a complex number from a $K$-theory class $\xi$ in $K^{\mathrm{alg}}_p(\mathcal{B}M)$ and a class $x$ in $K^{p-1}(\partial_h M)$, and we may ask whether these two numbers are equal:
\begin{equation}
\label{PairingEquation}
\big\langle \chi \ch(\xi), T\ch(x)\big\rangle ~\stackrel{?}{=}~ \langle \iota_* \xi, x \rangle.
\end{equation}
In the term $\chi \ch(\xi)$ on the left hand side, $\ch$ denotes the Chern character from algebraic $K$-theory to periodic cyclic homology and $\chi$ denotes the so-called coarse character map essentially introduced by Roe \cite{roe_coarse_cohomology}, which maps (periodic) cyclic homology of $\mathcal{B}M$ to coarse homology.
On the right hand side, 
\[
\iota_* \colon K_*^{\mathrm{alg}}(\cB M) \longrightarrow K_*^{\mathrm{top}}(C^*M)
\]
denotes the canonical comparison map.
The following result was proved by Roe in \cite{roe_coarse_cohomology}.

\begin{thmnonumber}[{\cite[Proposition 5.29]{roe_coarse_cohomology}}]
Equation~\cref{PairingEquation} holds provided that
\begin{enumerate}
\item $M$ is a complete Riemannian manifold;
\item the class $x$ is the pull-back of a class defined on a Higson-dominated, metrizable corona $N$ of $M$ which is homeomorphic to a finite polyhedron;
\item $\xi$ is the coarse index class of a graded generalized Dirac operator on $M$ (in the sense of \cite{gromov_lawson_complete}).
\end{enumerate}
\end{thmnonumber}

Roe posed the question whether \eqref{PairingEquation} holds in general \cite[Remark after Proposition 5.29]{roe_coarse_cohomology}, but to our knowledge, no answer to this question is available in the literature.
In concrete examples, the left hand side of \eqref{PairingEquation} is given as an operator trace and a priori takes values in complex numbers, while the right hand side of \eqref{PairingEquation} can be written as the index of an operator, which takes integer values. 
The equality of both is a ``quantization result'' for the trace-pairing; see \cite{LudewigThiangI}, where equality of \eqref{PairingEquation} is proved for general proper metric spaces and arbitrary classes in $K_*^\alg(\cB M)$, but only specific coarse cohomology classes.

\medskip

The purpose of this note is to show the following results. 
The first one is a variant of the above mentioned result of John Roe and introduces one of the main players of this note, the algebraic assembly map, which is defined in \S\ref{Assembly}.

\begin{thmA}
Suppose that $M$ is a complete Riemannian manifold and let $N$ be a Higson-dominated corona of $M$ such that $(M_N, N)$ is a finite CW-pair ($M_N$ is the compactification of $M$ by $N$).

Then equation \eqref{PairingEquation} holds for each class $\xi$ in $K_*^\alg(M)$, $* \in \{0, 1\}$, that lies in the image of the algebraic assembly map
\[
A^{\alg} \colon \left\{\substack{\text{\normalfont 1-summable}\\ 
	\text{\normalfont even/odd}\\\text{\normalfont Fredholm} \\  \text{\normalfont modules}}\right\}\longrightarrow K^{\alg}_*(\cB M)
\]
 and each class $x$ in $K^{*-1}(\partial_h M)$ that is obtained by pullback from $N$.
\end{thmA}

Our second theorem complements the mentioned results from \cite{roe_coarse_cohomology,LudewigThiangI} by giving conditions on the space $M$ such that \eqref{PairingEquation} holds.

\begin{thmB}
Suppose that $M$ is a complete Riemannian manifold and let $N$ be a Higson-dominated corona of $M$ such that $(M_N, N)$ is a finite CW-pair ($M_N$ is the compactification of $M$ by $N$).
Suppose moreover that the Baum--Connes assembly map
\[
A \colon K_*^{\mathrm{lf}}(M) \longrightarrow K_*^{\mathrm{top}}(C^*M)
\]
is surjective and that the comparison map
\[
\KH_*(\cB M) \longrightarrow K_*^{\mathrm{top}}(C^*M)
\]
from homotopy invariant $K$-theory of $\cB M$ to topological $K$-theory of the Roe algebra is injective.

Then equation \eqref{PairingEquation} holds for any $\xi$ in $K_*^{\alg}(\cB M)$, $* \in \{0, 1\}$, and any $x$ obtained by pullback from~$N$.
\end{thmB}

Applying a recent result of Bunke--Engel \cite{coarse_KH}, we obtain that the above conditions are satisfied, for example, for $M = \R^n$.

\begin{thmC}
Let $N$ be a metrizable, Higson-dominated corona of $\R^n$ such that $(\R^n_N, N)$ is a finite CW-pair.

Then equation \eqref{PairingEquation} holds for any $\xi$ in $K_*^{\alg}(\cB \R^n)$, $* \in \{0, 1\}$, and any $x$ obtained by pullback from~$N$.
\end{thmC}

Note that Theorem B (and consequently Theorem C) also hold more generally for the algebra $\cB^p M$ of finite propagation, locally $p$-summable operators (Remark~\ref{rem_nilinvariance}).

\subsubsection*{Outline of the argument.}

Let $(M_N, N)$ be a finite CW-pair and $N$ be Higson dominated. 
Then in Section~\ref{sec_transgress_main_diag} we reduce \eqref{PairingEquation} to commutativity of the diagram
\begin{equation}
\label{eq_big_diag_intro}
\begin{tikzcd}[column sep=3cm]
	K^{\mathrm{top}}_*(C^*M)
	\ar[r, "\tau_N"]
		&
		\tilde{K}_{*-1}^\an(N)
\\
	K^{\mathrm{alg}}_*(\mathcal{B}M)
	\ar[u, "\iota_*"]
	\ar[d, "\ch"']
		&
		\tilde{K}_{*-1}(N)
		\ar[dd, "\ch"]
		\ar[u, "\cong"]
\\
	\HP_*(\mathcal{B}M)
	\ar[d, "\chi"']
\\
	\HX^{\mathrm{per}}_*(M)
	\ar[r, "T_N"']
	&
		\tilde{H}_{*-1}^{\mathrm{per}}(N)
\end{tikzcd}
\end{equation}
stating that the two maps from $K^{\mathrm{alg}}_*(\mathcal{B}M)$ to $\tilde{H}_{*-1}^{\mathrm{per}}(N)$ in the diagram coincide. 

Unfortunately, we cannot establish commutativity of the diagram \eqref{eq_big_diag_intro} unconditionally.
Instead, in the Section~\ref{sec_reduction_elliptic}, under the additional assumption that $M$ is a smooth manifold, we make a further reduction to $1$-summable Fredholm modules by using the algebraic assembly map which we discuss in Section~\ref{sec_alg_assembly}.
This leads to Theorem A, which appears as Theorem~\ref{ThmMain1} in the main text.

Again, unfortunately, it is not clear if the assumptions of this intermediate theorem (that all classes in $K^\alg_*(\cB M)$ come from $1$-summable Fredholm modules) are always satisfied, hence we introduce Weibel's $\KH$-theory as a further tool. 
This step is based on the following diagram which can be grafted into the above one:
\begin{equation}
\label{eq_big_diag_2_intro}
\begin{tikzcd}
&
	K^{\mathrm{top}}_*(C^*M)
	&
	&
		K^{\mathrm{an},\mathrm{lf}}_*(M)
		\ar[ll, "A"']
		\ar[ddd, "\substack{\text{Connes--} \\ \text{Chern} \\ \text{character}}"]
\\
\KH_*(\mathcal{B}M)
\ar[ur]
\ar[dr]
	&
	K^{\mathrm{alg}}_*(\mathcal{B}M)
	\ar[u, "\iota_*"]
	\ar[l]
	\ar[d, "\ch"']
	&
	\left\{\substack{\text{1-summable}\\ 
	\text{even/odd}\\\text{Fredholm} \\  \text{modules}}\right\}
	\ar[ur, "{[-]}"']
	\ar[l, "A^{\mathrm{alg}}"']
\\
&
	\HP_*(\mathcal{B}M)
	\ar[d, "\chi"']
\\
&
	\HX^{\mathrm{per}}_*(M)
	&
	&
		H_*^{\dR, \mathrm{lf}, \mathrm{per}}(M),
		\ar[ll, "c"']
\end{tikzcd}
\end{equation}
This leads to Theorem B and consequently Theorem C.

\paragraph{Acknowledgements}
The first named author acknowledges financial support by the Deutsche Forschungsgemeinschaft DFG through Priority Programme SPP 2026 ``Geometry at Infinity'' (EN 1163/5-1, project number 441426261, \emph{Macroscopic invariants of manifolds}).
The second named author would like to thank SFB 1085 ``Higher Invariants'' for financial support. We thank Bernhard Hanke for a careful reading of the first version of this article.

\section{Preliminaries}
\label{sec_preliminaries}

\subsection{The index pairing and the analytic transgression map}
\label{sec_index_pairing}

Let $M$ be a proper metric space.
We start by discussing the index pairing
\begin{equation}
\label{IndexPairing}
K_*^{\mathrm{top}}(C^*M) \times \tilde{K}^{*-1}(\partial_hM) \longrightarrow \Z,
\end{equation}
which is, equivalently, a map
\[
K^{\mathrm{top}}_*(C^*M) \longrightarrow \Hom\big(\tilde{K}^{*-1}(\partial_h M), \Z\big).
\]
Here $\partial_h M$ is the Higson corona of $M$, defined as follows:
\begin{defn}\label{defn_higson_corona}
The Higson corona $\partial_h M$ of $M$ may be obtained as the Gelfand space of the quotient $C^*$-algebra $C_h(M)/C_0(M)$, where $C_h(M)$ is the algebra of continuous functions on $M$ with vanishing variation at infinity (see \cite[Chapter~5.1]{roe_coarse_cohomology}).

Its $K$-theory is, by definition, the $K$-theory of the $C^*$-algebra $C(\partial_h M) = C_h(M)/C_0(M)$. 
\end{defn}

$C^*M$ in \eqref{IndexPairing} is the Roe algebra of $M$, which is the norm closure of the algebra $\cB M$, consisting of all finite propagation, locally trace-class operators, acting on a separable Hilbert space $H$ with an ample representation of $C_0(M)$. In the following we will use that by Borel functional calculus we can extend this representation to a representation of the $C^*$-algebra $C_b(M)$ of all continuous, bounded functions on $M$.

\medskip

The pairing \eqref{IndexPairing} may be defined as follows.
Consider the map
\[
C^*M \otimes C(\partial_h M) \longrightarrow C^*M/\mathbb{K}(H), \qquad T \otimes [f] \longmapsto [T\tilde{f}],
\]
where we lift $f \in C(\partial_h M) = C_h(M)/C_0(M)$ to an element $\tilde{f}$ of $C_h(M)$, which is a well-defined $*$-homomorphism by the fact that elements of $C_h(M)$ have compact commutator with elements of $C^*M$ \cite[Proposition~4.1]{qiao_roe}, \cite[Lemma~3.9]{bunke2024coronascalliastypeoperators}. 
The index pairing \eqref{IndexPairing} is then obtained by applying the $K_1$-functor to this $*$-homomorphism, precomposing with the exterior product map and postcomposing with the $K$-theory boundary map to $K_0(\mathbb{K}(H)) \cong \Z$.
This explains the right hand side of \eqref{PairingEquation}.

\medskip

Suppose now that $N$ is a metrizable, Higson dominated corona of $M$, which means that there exists a continuous surjection $\partial_h M \to N$. 
Write $M_N \coloneqq M \sqcup_{\partial_h M} N$ for the corresponding compactification of $M$ by $N$.
The ample representation of $C_0(M)$ on $H$ may then be extended to a representation of $C(M_N)$. 
By $C^*_N M$ we denote the $C^*$-algebra generated by the locally compact operators on $H$ with controlled propagation for the coarse structure $\cC_N$ induced by the compactification $M_N$, see \cite[Exercise 6.1.10]{higson_roe}. 
Because the corona $N$ is Higson dominated, the coarse structure $\cC_N$ is coarser than the metric one and hence we have an inclusion $\iota_{\cC}\colon C^* M \subset C^*_N M$. This provides a map $K_*^{\mathrm{top}}(\iota_{\cC})\colon K_*^{\mathrm{top}}(C^* M) \to K_*^{\mathrm{top}}(C^*_N M)$.

The significance of $C^*_N M$ is that it coincides with the relative dual algebra $\mathfrak{D}(C(M_N)/\!\!/C_0(M))$ and the $K$-theory of the latter is isomorphic to the reduced analytic $K$-homology $\tilde{K}^{\mathrm{an}}_{*-1}(N)$ of the corona $N$, see \cite[Thm.\ 6.5.1 \& Cor.\ 6.5.2]{higson_roe}. 
(In \cite[Thm.\ 5.16]{roe_coarse_cohomology} this result is attributed to the never published preprint \cite{higson_never} of Higson.) 
Hence we have an isomorphism $\tilde\tau_N\colon K_*^{\mathrm{top}}(C^*_N M) \cong \tilde{K}_{*-1}^{\mathrm{an}}(N)$. 
Note that this result uses metrizability of $N$ (at least, the argument of Higson--Roe needs it).
Combining this with the previous map, we obtain a map
\[
\tau_N \colon K_*^{\mathrm{top}}(C^*M) \xrightarrow{K_*^{\mathrm{top}}(\iota_{\cC})} K_*^{\mathrm{top}}(C^*_NM) \stackrel{\tilde\tau_N}\cong \tilde{K}_{*-1}^{\mathrm{an}}(N) \cong \tilde{K}_{*-1}(N).
\]

\begin{lem}
\label{LemmaIndexPairing}
For any metrizable, Higson-dominated corona $N$ of $M$, we have a commutative diagram
\[
\begin{tikzcd}
K_*^{\mathrm{top}}(C^*M)
\ar[rr, "{\text{\normalfont index pairing}}"]
\ar[d, "K_*^{\mathrm{top}}(\iota_{\cC})"']
\ar[dr, "\tau_N"]
&
&
\Hom\big(\tilde{K}^{*-1}(\partial_h M), \Z\big)
\ar[d]
\\
K_*^{\mathrm{top}}(C^*_NM) 
\ar[r, "\tilde\tau_N"]
&
\tilde{K}_{*-1}^{\mathrm{an}}(N)  
\ar[r]
&
\Hom\big(\tilde{K}^{*-1}(N), \Z\big).
\end{tikzcd}
\]
\end{lem}

\begin{proof}
Let $A$ be a separable $C^*$-algebra, with a non-degenerate representation on a Hilbert space $H$ such that no non-zero $a \in A$ acts as a compact operator and let $J \subseteq A$ be an ideal.
Then the relative dual algebra $\mathfrak{D}(A/\!\!/J)\subseteq \mathbb{B}(H)$ consists of those bounded operators $T$ on $H$ such that $[T, a]$ is compact for all $a \in A$ and such that $Ta$ is compact for all $a \in J$.
It comes with a $*$-homomorphism
\[
\mathfrak{D}(A/\!\!/J) \otimes A/J \longrightarrow \mathbb{B}(H)/\mathbb{K}(H), \qquad T \otimes [a] \longmapsto Ta,
\]
which after applying $K$-theory yields an index pairing
\[
K_*^{\mathrm{top}}(\mathfrak{D}(A/\!\!/J)) \times K_{*-1}^{\mathrm{top}}(A/J) \longrightarrow \Z
\]
and consequently a map
\[
K_*^{\mathrm{top}}(\mathfrak{D}(A/\!\!/J)) \longrightarrow \Hom(K_{*-1}^{\mathrm{top}}(A/J), \Z).
\]
We apply this for $A = C(M_N)$ and $J = C_0(M)$, which results in $A/J\cong C(N)$, hence $K_{*-1}^{\mathrm{top}}(A/J) = K^{*-1}(N)$, and, as noted above, $\mathfrak{D}(A/\!\!/J) = C^*_N M$. 
Comparing this with the definition of the index pairing on the level of $C^*M$ explained above yields the desired commutativity.
\end{proof}

Analytic $K$-homology satisfies strong excision, which results in an isomorphism $K^\an_*(M_N,N) = K^{\an,\mathrm{lf}}_*(M)$. The boundary map in the long exact sequence for the pair $(M_N,N)$ is therefore a map $\partial\colon K_*^{\an,\lf}(M) \to K^\an_{*-1}(N)$ whose image is contained in the reduced $K$-homology $\tilde{K}_{*-1}^\an(N)$.
One then has the following result:

\begin{lem}[{\cite[Rem.\ 12.3.8]{higson_roe}}]
\label{lem_boundary_factors1}
We have a commutative diagram
\begin{equation*}
\begin{tikzcd}
  & K_*^{\mathrm{top}}(C^*M) \arrow[rr, "\tau_N", bend left] 
    & 
K_*^{\mathrm{an},\lf}(M) \arrow[l, "A"]
\arrow[r, "\partial"'] 
      &  
    \tilde{K}_{*-1}^\an(N),
  & 
\end{tikzcd}
\end{equation*}
where $A$ is the analytic coarse assembly map.
\end{lem}

\subsection{The algebraic Chern character and the coarse character map}
\label{sec_coarse_character}

Let $M$ be still a proper metric space.
The algebraic Chern character is abstractly defined as a natural transformation
\[
\ch \colon K^{\mathrm{alg}}_*(\cB M) \longrightarrow \HP_*(\cB M)
\]
from algebraic $K$-theory to periodic cyclic homology (see \cite[\S5.1]{loday_cyclic_hom}).
Recall that there are natural comparison maps
\begin{equation*}
\HP_*(\cB M) \longrightarrow \HC_{*+2k}(\cB M), \rlap{\qquad $k \in \Z$.}
\end{equation*}
In degrees $* \in \{0, 1\}$, the postcomposition of the Chern character with these maps allows explicit formulas; they are
\begin{align*}
\ch([P]) 
&= \frac{(2k)!}{k!}\cdot \big[\Tr(\underbrace{P\otimes \cdots \otimes P}_{2k+1})\big] \in \HC_{2k}(\cB M)
\\
\ch([U])&= k!\cdot \big[\Tr(\underbrace{(U-1) \otimes (U^{-1} - 1)\otimes \cdots }_{k+1~\text{pairs}})\big] \in \HC_{1 + 2k}(\cB M).
\end{align*}

We will use the following picture of coarse homology: $n$-chains are given by locally finite, signed Borel measures on $M^{n+1}$ that are supported in an $r$-neighborhood of the diagonal, endowed with the Alexander--Spanier differential; this yields the chain complex $C\!X_*(M)$ of coarse cochains.
One then has coarse character maps
\begin{equation}
\label{CoarseCharacterMap}
\chi_n \colon \HC_n(\cB M) \longrightarrow \HX_n(M),
\end{equation}
see \cite{LudewigThiangI}.
On the level of cyclic chains, $\chi_n(T_0 \otimes \cdots \otimes T_n)$ is the unique signed measure on $M^{n+1}$ 
with the property that for $f_0, \dots, f_n \in C_c(M)$, one has the formula 
\[
\int_M \cdots \int_M f_0(x_0) \cdots f_n(x_n) d\chi(T_0 \otimes \cdots \otimes T_n)(x_0, \dots x_n) = \Tr(f_0 T_0 \cdots f_n T_n).
\]

We define periodic version of the coarse character map
\[
\chi \colon \HP_*(\cB M) \longrightarrow \HX^{\mathrm{per}}_*(M) := \prod_{k \in \Z} \HX_{* + 2k}(M)
\]
which involves powers of $2 \pi i$; 
it is given as the product of the compositions
\[
\begin{tikzcd}
\HP_*(\cB M)
\ar[r]
&
\HC_{*+2k}(\cB M)
\ar[r, "(2 \pi i)^{-k}\cdot \chi_{*+2k}"]
&[2cm]
\HX_{*+2k}(M).
\end{tikzcd}
\]
We conclude that for $*\in \{0, 1\}$ the composition
\[
\begin{tikzcd}
K^{\mathrm{alg}}_*(\cB M) \ar[r]
&
 \HP_*(\cB M) \ar[r, "\chi"]
 & 
 \HX_{*}^{\mathrm{per}}(M)
\end{tikzcd}
\]
in the factor labelled by $k \in \Z$ is given by
\[
\int_{M^{2k+1}} f_0 \otimes \cdots \otimes f_{2k}\, d\ch([P]) 
= \frac{(2k)!}{(2 \pi i)^k k!}	\cdot \chi(f_0 P \cdots f_{2k} P)
\]
in the even case and by
\[
\int_{M^{2k+2}} f_0 \otimes \cdots \otimes f_{2k+1}\, d\ch([U]) 
= \frac{k!}{(2 \pi i)^k}	\cdot \chi\big(f_0 (U-1) f_1 (U^{-1} - 1) \cdots  f_{2k+1} (U^{-1} - 1)\big)
\]
in the odd case.

\subsection{Analytic and topological Chern characters, and boundary maps}
\label{sec_chern_boundary}

In this section the assumption will come in that $M$ is a smooth manifold and that it is the interior of a finite CW-pair $(M_N,N)$, where $M_N$ is any compactification of $M$ by a corona $N$. This finiteness assumption is due to \cite{baum_higson_schick} which we are going to use.

Let us discuss the different pictures for $K$-homology and the respective Chern characters on them. The necessity for changing pictures is seen in Diagram~\eqref{eq_big_diag_intro}: The Chern character on the left hand side is defined analytically by computing traces of operators and we have to compare it with a Chern character on the corona $N$ of $M$. 
Assuming $N$ to be a (compact) manifold we could use the analytic Chern character on $\tilde{K}^\an_{*-1}(N)$ and therefore stay completely in the analytic world to prove commutativity of \eqref{eq_big_diag_intro}. 
Since we do not want to restrict ourselves even further by assuming that $N$ is a smooth manifold, we are forced to use a different picture for the $K$-homology of $N$ and the respective Chern character.

There are two other pictures of $K$-homology: The topological one, which is the homology theory associated to the $K$-theory spectrum, and the geometric one, introduced by Baum--Douglas \cite{baum_douglas}. 
Each comes with a comparison map to analytic $K$-homology which is known to be an isomorphism whenever $(M_N, N)$ is a finite CW-pair \cite{baum_higson_schick}. 
Each of these two pictures of $K$-homology also comes with a naturally defined Chern character on it to the periodized homology of $N$. Since for our purposes both of these two pictures of $K$-homology do the job, we just pick one once and for all and denote it by $\tilde{K}_{*-1}(N)$. 
This explains the upper right vertical isomorphism in Diagram~\eqref{eq_big_diag_intro} and our necessity for assuming that $(M_N, N)$ is a finite CW-pair.

Let us discuss the Connes--Chern character which is the Chern character on $K$-homology in the analytic picture \cite{connes_noncomm_diff_geo} (from here on we need in this discussion that $M$ is a smooth manifold): Analytic $K$-homology is defined in terms of Fredholm modules, but the character is only defined for finitely summable ones (i.e., $p$-summable for some $p\in [1,\infty)$), i.e. we have a map $K^{\an,\lf,\fin}_*(M) \to H^{\dR,\lf,\per}_*(M)$ to locally finite, periodized de Rham homology of $M$. 
By the localized index formula of Connes--Moscovici \cite[Theorem~3.9]{connes_moscovici}, one concludes that the Connes--Chern character of the finitely summable Fredholm module defined by a graded generalized Dirac operator on $M$ (in the sense of \cite{gromov_lawson_complete}) coincides with the classical Chern character of the $K$-homology class of that operator \cite[Theorem~4.6]{connes_moscovici}. By this together with the fact that (rationally) every $K$-homology class of $M$ may be represented by the class of a graded generalized Dirac operator on $M$ (cf.~footnote~\ref{footnote_summable}), we infer that the Connes--Chern character descends to $K^{\an,\lf}_*(M)$ and we have a commutative diagram
\begin{equation}
\label{eq_comp_Chern}
\begin{tikzcd}
		K^{\mathrm{an}, \mathrm{lf}}_*(M)
		\ar[dd, "\substack{\text{Connes--} \\ \text{Chern} \\ \text{character}}"']
		&[-0.3cm]
		K^{\mathrm{lf}}_*(M)
		\ar[dd, "\ch"]
		\ar[l]
\\
\\
	H_*^{\mathrm{dR}, \mathrm{lf}, \mathrm{per}}(M)
	&
		H_*^{\mathrm{lf}, \mathrm{per}}(M)
		\ar[l]
\end{tikzcd}
\end{equation}

Let $M_N$ be any compactification of $M$ by a corona $N \coloneqq M_N \setminus M$. 
As discussed in Section~\ref{sec_index_pairing} that analytic $K$-homology satisfies strong excision, and therefore we have an isomorphism $K_*^\an(M_N,N) \cong K_*^{\an,\lf}(M)$, the boundary map in the long exact sequence in analytic $K$-homology for the pair $(M_N,N)$ identifies with a map $\partial\colon K_*^{\an,\lf}(M) \to \tilde{K}_{*-1}^\an(N)$.
Under the assumption that $(M_N,N)$ is a finite CW-pair we can then switch pictures and get a boundary map $\partial\colon K_*^{\lf}(M) \to \tilde{K}_{*-1}(N)$. 
 Now we use the fact that the Chern character is a natural transformation of cohomology theories to get a commutative diagram
\begin{equation}
\label{BigDiagram_rightpartsmall}
\begin{tikzcd}
		K^{\mathrm{lf}}_*(M)
		\ar[r, "\partial"]
		\ar[dd, "\ch"]
		&[-0.3cm]
		\tilde{K}_{*-1}(N)
		\ar[dd, "\ch"]
\\
\\
		H_*^{\mathrm{lf}, \mathrm{per}}(M)
		\ar[r, "\partial"]
		&
			\tilde{H}_{*-1}^{\mathrm{per}}(N),
\end{tikzcd}
\end{equation}
where the bottom map is the boundary map in periodized homology.

\medskip

Combining all the above results in the following:

\begin{prop}\label{prop_different_cherns}
Let $M$ be a smooth manifold and assume that it is the interior of a finite CW-pair $(M_N,N)$. Then we have a commutative diagram
\begin{equation*}
\label{BigDiagram_rightpart}
\begin{tikzcd}
		K^{\mathrm{an}, \mathrm{lf}}_*(M)
		\ar[dd, "\substack{\textup{Connes--} \\ \textup{Chern} \\ \textup{character}}"']
		&[-0.3cm]
		K^{\mathrm{lf}}_*(M)
		\ar[r, "\partial"]
		\ar[dd, "\ch"]
		\ar[l, "\cong"']
		&[-0.3cm]
		\tilde{K}_{*-1}(N)
		\ar[dd, "\ch"]
\\
\\
	H_*^{\mathrm{dR}, \mathrm{lf}, \mathrm{per}}(M)
	&
		H_*^{\mathrm{lf}, \mathrm{per}}(M)
		\ar[r, "\partial"]
		\ar[l, "\cong"']
		&
			\tilde{H}_{*-1}^{\mathrm{per}}(N).
\end{tikzcd}
\end{equation*}
\end{prop}

\section{Main argument}

In this section, we explain the main diagram used for the proof.

\subsection{Transgression and the main diagram}
\label{sec_transgress_main_diag}

We start our discussion here with the assumption that $M$ is a general proper metric space.

The Higson corona supports transgression maps
\begin{equation}
\label{TransgressionMapsHigsonCorona}
T \colon \tilde{H}^{*-1}(\partial_h M) \longrightarrow \HX^{*}(M), \qquad T \colon \HX_{*}(M) \longrightarrow \tilde{H}_{*-1}(\partial_h M),
\end{equation}
relating Alexander--Spanier (co)homology of the Higson corona $\partial_h M$ with the coarse (co)homology of $M$, see \cite[Section~\S4.3]{engel_wulff}; we remark here that it is important to use Alexander--Spanier (co)homology as $\partial_h M$ is usually not homotopy equivalent to a CW-complex.

The idea to define these transgression maps is as follows: Let us first note that $\HX^*(M) \cong H^*(\cP(M))$ and $\HX_*(M) \cong H_*^\lf(\cP(M))$, where $\cP(M)$ is the total Rips complex of $M$ considered as a $\sigma$-space, which means that we remember the sequence of Rips complexes $P_R(M)$ for $R \in \IN$ assembling to $\cP(M)$, and the functors $H^*(-)$ and $H_*^\lf(-)$ are correspondingly extended to $\sigma$-spaces. 
Because Alexander--Spanier (co)homology satisfies strong excision, we have $H^*(\cP(M)) \cong H^*(\overline{\cP(M)},\partial_h M)$ and $H_*^\lf(\cP(M)) \cong H_*(\overline{\cP(M)},\partial_h M)$, where $\overline{\cP(M)}$ denotes the Higson compactification of $\cP(M)$.\footnote{This is abuse of terminology: The space $\cP(M)$ is usually not locally compact and hence $\overline{\cP(M)} = \cP(M) \sqcup \partial_h M$ is not compact.}
The transgression maps are then the boundary maps in the long exact sequences of the pair $(\overline{\cP(M)},\partial_h M)$ together with the observation that they factor through (resp. have image contained in) the corresponding reduced groups.

The canonical map $M \to \cP(M)$ yields a canonical comparison map
\[
c \colon H^{\lf}_*(M) \longrightarrow \HX_*(M).
\] 
Naturality of boundary maps shows that the transgression map fits into the commutative diagram
\begin{equation}
\begin{tikzcd}
\label{BoundaryComparisonMap}
H^{\lf}_*(M) 
\ar[d, "\partial"']
\ar[r, "c"]
& \HX_*(M)
\ar[d, "T"]
\\
\tilde{H}_{*-1}(\partial_h M) 
\ar[r, equal]
&
\tilde{H}_{*-1}(\partial_h M),
\end{tikzcd}
\end{equation}
where the left hand side is the boundary map for the pair $(\overline{M}, \partial_h M)$ with $\overline{M}$ being the Higson compactification of $M$.

The transgression maps satisfy the duality
\[
\langle Tx, \xi\rangle = \langle x, T\xi\rangle
\]
for chains $x \in \HX_p(M)$ and cochains $\xi \in \tilde{H}^{p-1}(\partial_hM)$, where the left hand side denotes the pairing in Alexander--Spanier (co)homology of $\partial_hM$ and the right hand side denotes the pairing in coarse (co)homology of $M$.
Using this, we may move the transgression on the left hand side of \eqref{PairingEquation} to obtain
\begin{equation}
\label{ShiftTransgression}
\langle \chi \ch(\xi), T \ch(x) \rangle = \langle T\chi \ch(\xi), \ch(x) \rangle .
\end{equation}

\medskip

If now $N$ is a Higson dominated corona, there are transgression maps (which we denote by $T_N$) for $N$ that factor through \eqref{TransgressionMapsHigsonCorona} and analogously to the diagram \eqref{BoundaryComparisonMap}, we have a commutative diagram
\begin{equation}
\begin{tikzcd}
\label{BoundaryComparisonMap_N}
H^{\lf}_*(M) 
\ar[d, "\partial"']
\ar[r, "c"]
& \HX_*(M)
\ar[d, "T_N"]
\\
\tilde{H}_{*-1}(N) 
\ar[r, equal]
&
\tilde{H}_{*-1}(N),
\end{tikzcd}
\end{equation}

Provided $(M_N, N)$ is a finite CW-pair (so that we have the isomorphism in the upper right corner of \eqref{SmallerDiagram}), we may now ask whether the diagram
\begin{equation}
\label{SmallerDiagram}
\begin{tikzcd}[column sep=3cm]
	K^{\mathrm{top}}_*(C^*M)
	\ar[r, "\tau_N"]
		&
		\tilde{K}_{*-1}^\an(N)
\\
	K^{\mathrm{alg}}_*(\mathcal{B}M)
	\ar[u, "\iota_*"]
	\ar[d, "\ch"']
		&
		\tilde{K}_{*-1}(N)
		\ar[dd, "\ch"]
		\ar[u, "\cong"]
\\
	\HP_*(\mathcal{B}M)
	\ar[d, "\chi"']
\\
	\HX^{\mathrm{per}}_*(M)
	\ar[r, "T_N"']
	&
		\tilde{H}_{*-1}^{\mathrm{per}}(N)
\end{tikzcd}
\end{equation}
commutes, i.e., if the two maps from $K^{\mathrm{alg}}_*(\mathcal{B}M)$ to $\tilde{H}_{*-1}^{\mathrm{per}}(N)$ that we see in the diagram coincide.

\begin{lem}
Suppose that $(M_N, N)$ is a finite CW-pair and $N$ is Higson dominated. 

Then if the diagram \eqref{SmallerDiagram} commutes, \eqref{PairingEquation} holds for each class $\xi$ in $K^{\mathrm{alg}}_*(\cB M)$ and each class $x$ of the form $x = p^*x'$ for some class $x' \in \tilde{K}^{*-1}(N)$ with $p\colon \partial_h M \to N$ the domination map.
\end{lem}

\begin{proof}
By \cref{LemmaIndexPairing}, if $x \in \tilde{K}^{p-1}(\partial_h M)$ is the pullback of a class $x' \in \tilde{K}^{p-1}(N)$ under $p\colon \partial_h M \to N$, the right hand side of \eqref{PairingEquation} may be reformulated as
\begin{equation}
\label{RewritingRightHandSide}
\langle \iota_* \xi, x \rangle = \big\langle \tau_N(\iota_*\xi), x' \rangle,
\end{equation}
where the right hand side is the canonical pairing between (reduced) $K$-theory and $K$-homology of $N$. 
With a view on \eqref{RewritingRightHandSide}, the question becomes under which assumptions we can prove the equality of 
\begin{equation}
\label{PairingEquation2}
\big\langle \tau_N(\iota_*\xi), x' \rangle ~\stackrel{!}{=}~ \big\langle T_N\chi \ch(\xi), \ch(x')\big\rangle = \big\langle \ch\big(\tau_N(\iota_*\xi)\big), \ch(x')\big\rangle,
\end{equation}
where in the second step, we used commutativity of \eqref{SmallerDiagram}.
We may now ``cancel'' the Chern characters on the right hand side of \eqref{PairingEquation2}, which means to exploit that the pairing of a $K$-homology class with a $K$-theory class is the same as the pairing of the corresponding classes in (co-)homology obtained by applying the Chern characters, to obtain the left hand side, thus showing equality in \eqref{PairingEquation2} and \eqref{PairingEquation}.
\end{proof}

\subsection{Reduction to $1$-summable Fredholm modules}
\label{sec_reduction_elliptic}

We will not be able to establish the commutativity of \eqref{SmallerDiagram} unconditionally.
Our discussion is based on the following inflation of \eqref{SmallerDiagram}, where we have to assume that $N$ is Higson dominated, and $(M_N, N)$ is a finite CW-pair and $M$ is a smooth manifold (since we will use the discussion in Section \ref{sec_chern_boundary}):
\begin{equation}
\label{BigDiagram}
\begin{tikzcd}
&[-0.1cm] &[-1.1cm] &[-0.4cm] &[-0.4cm] \tilde{K}_{*-1}^\an(N)
\\
	K^{\mathrm{top}}_*(C^*M)
	\ar[rrrru, bend left=10, "\tau_N"]
	&
	&
		K^{\mathrm{an}, \mathrm{lf}}_*(M)
		\ar[ll, "A"']
		\ar[ddd, "\substack{\text{Connes--} \\ \text{Chern} \\ \text{character}}"]
		&
		K^{\mathrm{lf}}_*(M)
		\ar[r, "\partial"]
		\ar[ddd, "\ch"]
		\ar[l, "\cong"']
		&
		\tilde{K}_{*-1}(N)
		\ar[ddd, "\ch"]
		\ar[u, "\cong"']
\\
	K^{\mathrm{alg}}_*(\mathcal{B}M)
	\ar[u, "\iota_*"]
	\ar[d, "\ch"']
	&
		\left\{\substack{\text{1-summable}\\ 
	\text{even/odd}\\\text{Fredholm} \\  \text{modules}}\right\}
	\ar[ur, "{[-]}"']
	\ar[l, "A^{\mathrm{alg}}"']
\\
	\HP_*(\mathcal{B}M)
	\ar[d, "\chi"']
\\
	\HX^{\mathrm{per}}_*(M)
	\ar[rrrr, bend right=15, "T_N"]
	&
	&
	H_*^{\mathrm{dR}, \mathrm{lf}, \mathrm{per}}(M)
			\ar[ll, "c"']
	&
		H_*^{\mathrm{lf}, \mathrm{per}}(M)
		\ar[r, "\partial"]
		\ar[l, "\cong"']
		&
			\tilde{H}_{*-1}^{\mathrm{per}}(N)
\end{tikzcd}
\end{equation}
Here we consider the set of $1$-summable Fredholm modules, which on the one hand define elements in the analytic $K$-homology group $K_*^{\mathrm{an}, \mathrm{lf}}(M)$.
On the other hand, we will show that (in the case that $* \in \{0, 1\}$) there is an ``algebraic assembly map'' that yields elements in $K_*^{\mathrm{alg}}(\cB M)$, in such a way that both subdiagrams of the large left rectangle commute. 

Commutativity of the diagram involving $\tau_N$ is Lemma~\ref{lem_boundary_factors1} (together with a change of pictures for $K$-homology) and commutativity of the right two squares is Proposition~\ref{prop_different_cherns}. 
The diagram involving $T_N$ is \eqref{BoundaryComparisonMap_N}.
Regarding the left square,
the following result will be proved in Sections~\ref{sec_alg_assembly} and \ref{SectionConnesMoscovici}:

\begin{prop}\label{prop_both_subdiag_commute}
Both subdiagrams of the left square of \eqref{BigDiagram} commute.
\end{prop}

Notice, however, that the arrows in \eqref{BigDiagram} are oriented in such a way that the commutativity of all subdiagrams above does \emph{not} yield the commutativity of \eqref{SmallerDiagram}.
However, we may record the following consequence of Proposition \ref{prop_both_subdiag_commute}:

\begin{thm}
\label{ThmMain1}
Suppose that $N$ is Higson dominated, that $(M_N, N)$ is a finite CW-pair and that $M$ is a smooth manifold. 

Then Diagram~\eqref{SmallerDiagram} commutes on those classes in $K_*^{\mathrm{alg}}(\cB M)$ that lie in the image of the algebraic assembly map (i.e., come from a $1$-summable Fredholm module).
Consequently, \eqref{PairingEquation} holds for such classes and any class $x$ that is a pullback from $N$.
\end{thm}

Note that it is not clear (and we actually do not really expect it to be true) that \emph{all} classes in $K_*^{\mathrm{alg}}(\cB M)$ lie in the image of the algebraic assembly map $A^{\mathrm{alg}}$.

To move on with our argument, we consider the extended diagram
\begin{equation}
\label{BigDiagram2}
\begin{tikzcd}
&
	K^{\mathrm{top}}_*(C^*M)
	&
	&
		K^{\mathrm{an},\mathrm{lf}}_*(M)
		\ar[ll, "A"']
		\ar[ddd, "\substack{\text{Connes--} \\ \text{Chern} \\ \text{character}}"]
\\
\KH_*(\mathcal{B}M)
\ar[ur, "\eta_{\toprm}"]
\ar[dr]
	&
	K^{\mathrm{alg}}_*(\mathcal{B}M)
	\ar[u, "\iota_*"]
	\ar[l, "\eta_{\KH}"']
	\ar[d, "\ch"']
	&
	\left\{\substack{\text{1-summable}\\ 
	\text{even/odd}\\\text{Fredholm} \\  \text{modules}}\right\}
	\ar[ur, "{[-]}"']
	\ar[l, "A^{\mathrm{alg}}"']
\\
&
	\HP_*(\mathcal{B}M)
	\ar[d, "\chi"']
\\
&
	\HX^{\mathrm{per}}_*(M)
	&
	&
		H_*^{\dR, \mathrm{lf}, \mathrm{per}}(M),
		\ar[ll, "c"']
\end{tikzcd}
\end{equation}
where $\KH_*(\cB M)$ denotes Weibel's homotopy $K$-theory of the algebra $\mathcal{B} M$.
For our purposes, the important property of it is that both the comparison map $\iota_*$ and the Chern character $\ch$ factor through it.
An immediate consequence is the following:

\begin{thm}
\label{ThmMainThm2}
Suppose that $N$ is Higson dominated, $(M_N, N)$ is a finite CW-pair and that $M$ is a smooth manifold.
Assume moreover that
\begin{enumerate}
\item[{\normalfont (i)}] The assembly map $A \colon K^{\mathrm{an}, \mathrm{lf}}(M) \to K^{\mathrm{top}}_*(C^*M)$ is surjective;
\item[{\normalfont (ii)}] The comparison map $\eta_\toprm\colon \KH_*(\cB M) \to K^{\mathrm{top}}_*(C^* M)$ is injective.
\end{enumerate}
Then the diagram \eqref{SmallerDiagram} commutes for $* = \{0, 1\}$.
Consequently, \eqref{PairingEquation} holds for any class $\xi$ in $K^{\mathrm{alg}}_*(\cB M)$ and any class $x$ that is a pullback from $N$.
\end{thm}

\begin{proof}
Let $\xi \in K^\alg_*(\cB M)$.
By surjectivity of $A$, there exists a class $[T] \in K_*^{\mathrm{an}, \mathrm{lf}}(M)$ such that
\begin{equation}\label{eq_main_diag_1}
A([T]) = \iota_* \xi.
\end{equation}
Since (rationally) every locally finite $K$-homology class of $M$ may be represented by a summable Fredholm module,\footnote{This is because every locally finite $K$-homology class may be represented by the class of a graded, generalized Dirac operator and they define summable classes. The argument in the compact case is outlined in, e.g., \cite[last page of Appendix C]{lawson_michelsohn}. This argument extends to the locally finite case.\label{footnote_summable}} we may assume that $T$ is summable.\footnote{Unfortunately, there is a discrepancy here between the two occurrences of the word ``summable'': The representation result of $K$-homology classes allows us only to represent classes by $p$-summable modules with $p$ depending on the dimension of the manifold. But we define the algebraic assembly map $A^\alg$ only for $1$-summable modules. To resolve this issue, we observe that $A^\alg$ may also be defined on $p$-summable modules and maps them then to $K_*^\alg(\cB^p M)$, where $\cB^p M$ is the algebra of finite propagation, locally $p$-summable operators. Then we just exploit that the inclusion $\cB M \to \cB^p M$ induces an isomorphism on $\KH$-theory, see Remark~\ref{rem_nilinvariance}.}
Notice that we do \emph{not} know if $\xi = A^{\mathrm{alg}}(T)$, but from \eqref{eq_main_diag_1} and commutativity of \eqref{BigDiagram2} we get $\iota_*(A^\alg(T)) = \iota_*(\xi)$. 
Replacing $\iota_*$ by the composition $\eta_\toprm \circ \eta_{\KH}$ and using Assumption (ii) we get $\eta_{\KH}(A^\alg(T)) = \eta_{\KH}(\xi)$, which finally implies
\[
\ch(\xi) = \ch\big(A^{\mathrm{alg}}(T)\big).
\]
Hence from commutativity of the subdiagrams of \eqref{BigDiagram} we get 
\begin{align*}
T_N \chi\ch(\xi) 
&= T_N \chi\ch\big(A^{\mathrm{alg}}(T)\big)
\\
&= T_N c\big(\ch([T])\big) 
\\
&= \partial\big(\ch([T])\big)
\\
&= \ch(\partial [D])
\\
&= \ch\big(\tau_N(A [T])\big)
\\
&= \ch\big(\tau_N(\iota_* \xi)\big),
\end{align*}
which is exactly commutativity of diagram \eqref{SmallerDiagram}.
\end{proof}

A case where both assumptions of the above theorem are satisfied is the following.
It is well-known that the coarse assembly map is an isomorphism for $M = \R^n$, and a recent result of Bunke--Engel \cite{coarse_KH} implies that the comparison map $\KH_*(\cB M) \to K^{\mathrm{top}}_*(C^*M)$ is also an isomorphism for $M = \R^n$.
Briefly put, the reason for this latter claim is that both sides are coarse homology theories, so that the comparison map begin an isomorphism follows from the fact that it is an isomorphism on a point by results of Corti\~nas--Thom \cite{CT}. 

\begin{cor}
The assumptions of the Theorem~\ref{ThmMainThm2} are satisfied when $M = \R^n$ and $N$ is a Higson dominated, finite complex such that $(\R^n \sqcup N, N)$ is a finite CW-pair.
\end{cor}

The above corollary is, even though $\IR^n$ seems to be an extremely special case of the very general Theorem \ref{ThmMainThm2}, a new result: We don't expect that every class in $K^{\mathrm{alg}}_*(\cB \IR^n)$ is representable by a graded generalized Dirac operator on $\IR^n$ and therefore Roe's original result does not cover all cases here.

\begin{rem}\label{rem_nilinvariance}
Theorem~\ref{ThmMainThm2} also holds more generally for the algebra $\cB^p M$ of finite propagation, locally $p$-summable operators. The main reason for this is that we have an isomorphism $\KH_*(\cB M) \cong \KH_*(\cB^p M)$ due to nil-invariance of $\KH$-theory: the quotient $\cB^p M / \cB M$ is a nilpotent algebra and hence $\KH_*(\cB^p M / \cB M) \cong 0$. Of course, the results of Connes--Moscovici and Moscovici--Wu that we use are also valid for $p$-summable Fredholm modules and not just $1$-summable ones.
\end{rem}

\section{Algebraic assembly}
\label{Assembly}

In this section, we define our algebraic assembly map, which allows to compare the algebraic Chern character with the Connes--Chern character.

\subsection{Definition}
\label{sec_alg_assembly}

Let $M$ be a proper metric space.
In this section, we define a map
\[
A^{\alg} \colon \left\{\substack{\text{1-summable}\\ 
	\text{even/odd}\\\text{Fredholm} \\  \text{modules}}\right\} \longrightarrow K^{\alg}_*(\cB M),
\]
for $* \in \{0, 1\}$, which makes the diagram
\begin{equation}
\label{DiagramSection6}
\begin{tikzcd}
	K^{\mathrm{top}}_*(C^*M)
	&
	&
		K^{\mathrm{an},\mathrm{lf}}_*(M)
		\ar[ll, "A"']
\\
	K^{\mathrm{alg}}_*(\mathcal{B}M)
	\ar[u, "\iota_*"]
	&
	\left\{\substack{\text{1-summable}\\ 
	\text{even/odd}\\\text{Fredholm} \\  \text{modules}}\right\}
	\ar[ur, "{[-]}"']
	\ar[l, "A^{\mathrm{alg}}"']
\end{tikzcd}
\end{equation}
commutative.
We do not attempt to show that $A^{\alg}$ respects any equivalence relations one might put on the set of $1$-summable Fredholm modules.
Throughout, we fix an ample\footnote{Recall that this means that the representation of $C_0(M)$ on $H$ is non-degenerate and that no non-zero function from $C_0(M)$ acts as a compact operator.} $M$-module $H$, on which the algebra $\cB M$ of finite propagation, locally trace-class operators acts, and we will use the notation $\cA M$ for the algebra of finite propagation operators on $M$. Note that $\cA M$ contains $\cB M$ as an ideal.

We fix a dense *-subalgebra $\cA \subset C_c(M)$ containing partitions of unity subordinate to any finite open cover, as well as square roots of these functions. 
If $M$ is a manifold, then $\cA = C_c^\infty(M)$ is a valid choice.

\paragraph{\underline{The even case.}}

An even, $1$-summable Fredholm module is a bounded, self-adjoint operator $T$ on $H \oplus H$ having the form
\[
T = \begin{pmatrix}
	0 & T_- \\ T_+ & 0
\end{pmatrix},
\]
such that for each $f \in \cA$ the operators $f(T^2 - 1)$, $(T^2 - 1)f$ and $[T, f]$ are trace-class.

We choose a uniformly bounded and locally finite open cover $(U_i)_{i \in I}$ of $M$ and let $(\chi_i)_{i \in I}$ be a subordinate partition of unity from $\cA$.
Then we set
\[
\tilde{T}_+ \coloneqq \sum_{i \in I} \chi_i^{1/2} T_+ \chi_i^{1/2}.
\] 
This operator has finite propagation by uniform boundedness of the cover and is invertible modulo $\cB M$, hence we obtain an element $[\tilde{T}_+]$ in $K_1^{\alg}(\cA M/\cB M)$. 
Then we set
\[
A^{\alg}(T) \coloneqq \partial([\tilde{T}_+]) \in K_0^{\alg}(\cB M),
\]
where $\partial$ is the boundary map in algebraic $K$-theory.\footnote{The ideal inclusion $\cB(M) \subset \cA(M)$ does probably not satisfy excision in algebraic $K$-theory, so the boundary map generally sends $K_j(\cA(M)/\cB(M))$ to the relative $K_{j-1}(\cA(M), \cB(M))$. 
However, in degree zero, one always has $K_0(\cA(M), \cB(M)) \cong K_0(\cB(M))$ 
\cite[Thm.~1.5.9]{rosenberg_algebraic_K}.
Moreover, there exists an explicit description of the boundary map from $K_1$ to $K_0$ analogous to the index map in topological $K$-theory, which lands directly in $K_0(\cB(M))$.}
Another choice of cover $(U_i)_{i \in I}$ and subordinate partition of unity $(\chi_i)_{i \in I}$ leads to an operator $\tilde{T}^\prime_+$ which differs from $\tilde{T}_+$ by a finite propagation, locally trace-class operator and hence defines the same element in $K_1^{\alg}(\cA M/\cB M)$. Consequently, $A^{\alg}(T)$ is well-defined.

The diagram \eqref{DiagramSection6} commutes, because the usual assembly map $A$ in topological $K$-theory is defined by exactly the same procedure.

\paragraph{\underline{The odd case.}}

An odd, $1$-summable Fredholm module is a bounded, self-adjoint operator $T$ on $H$ such that for each $f \in \cA$ the operators $f(T^2 - 1)$, $(T^2 - 1)f$ and $[T, f]$ are trace-class.

Let us first recall the definition of the usual assembly map $A$ in this case. First we replace $T$ by the finite propagation operator
\begin{equation}\label{eq_truncation}
\tilde{T} \coloneqq \sum_{i \in I} \chi_i^{1/2} T \chi_i^{1/2}
\end{equation}
as in the case of even Fredholm modules. Then $P = (1-\tilde{T})/2$ is a projection in $\cA M / \cB M$ and we obtain a class $[P] \in K_0^\alg(\cA M / \cB M)$. 
Then we define
\begin{equation}\label{eq_defnA_odd}
A(T) \coloneqq \partial\left[P\right] \in K_1^{\mathrm{top}}(C^*M),
\end{equation}
where we pass to the completions of the algebras and $\partial$ denotes the boundary map in the $6$-term exact sequence of $K$-theory of $C^*$-algebras.

Note that $\partial\left[P\right]$ is usually represented by $[\exp(2 \pi i \cdot P)]$, but we have to deviate from this to define $A^\alg$, because in general $\exp(2 \pi i \cdot P)$ has infinite propagation. But since homotopic invertibles over $C^*M$ define the same class in $K_1^{\mathrm{top}}(C^*M)$, we have some flexibility here to use a different function $\varphi$ instead of $\exp(2 \pi i \cdot -)$. The rough idea is as follows: If $\varphi$ does not take the value $0$, then we will get by functional calculus an invertible element $\varphi(P)$, and if $\varphi$ is homotopic to $\exp(2 \pi i \cdot -)$ through functions which all do not take the value $0$, then $\varphi(P)$ will represent the same class in $K_1^{\mathrm{top}}(C^*M)$ as $[\exp(2 \pi i \cdot P)]$. Taking care of the details in this argument, especially taking care of the non-unitality of $C^*M$, we arrive at the following:

We consider continuous functions $\varphi \colon I \to \C$, where $I \subset \IR$ is a sufficiently large interval containing $[0,1]$ and the spectrum of $P$, with
\begin{enumerate}
\item $\varphi(0) = \varphi(1) = 1$;
\item $\varphi$ does not have any zeros; and
\item the restriction of $\varphi$ to $[0, 1]$ has winding number $+1$ around $0$.\footnote{This is a convenient way of phrasing that $\varphi$ is homotopic to $\exp(2 \pi i \cdot -)$ through functions which all do not take the value $0$.}
\end{enumerate}
By (i) and (ii) we infer that $\varphi(P) - 1$ is in $C^* M$ and is invertible; we obtain an element $[\varphi(P)] \in K_1^\toprm(C^* M)$. We observe that the function $\varphi_0(\lambda) = \exp(2 \pi i \lambda)$ satisfies all these conditions and that any function $\varphi$ satisfying (i), (ii) and (iii) may by homotoped to $\varphi_0$ through functions also satisfying (i), (ii) and (iii).
Hence $\varphi(P)$ represents the class $A(T)$ for any $\varphi$ with these properties.

It is a consequence of the Weierstra{\ss} approximation theorem that there in fact exists a \emph{polynomial} function $\varphi$ satisfying (i), (ii) and (iii).
One explicit choice is
\[
\varphi(\lambda) = 1 + (\lambda^2 - \lambda)\psi(\lambda), \qquad \text{with}\qquad \psi(\lambda) = 2i\left(t - \tfrac{1}{2}\right) + 5.
\]
For such a function $\varphi$, the operator $\varphi(P)$ has finite propagation.
Since $P^2 - P \in \cB M$, we conclude that $\varphi(P) - 1  = (P^2 - P)\in \cB M$.
Hence we may set
\[
A^{\alg}(T) \coloneqq [\varphi(P)] \in K_1^{\alg}(\cB M),
\]
which agrees with $A(T)$ after passing to topological $K$-theory by the above discussion.

\begin{rem}
Two remarks are in order about the odd case.
\begin{enumerate}
\item Our construction of $A^\alg(T)$ depends on the choice of the polynomial $\varphi$, and it is not clear to us whether one can prove that it is independent from this choice.

But what we can prove is the following: If $\varphi$ and $\varphi^\prime$ are two choices of polynomials satisfying (i)--(iii), then by an argument using the Weierstra{\ss} approximation theorem they will be \emph{(piece-wise) polynomially} homotopic to each other through polynomials satisfying (i)--(iii). Hence $[\varphi(P)]$ and $[\varphi^\prime(P)]$ will coincide after applying the comparison map $K_1^\alg(\cB M) \to \KH_1(\cB M)$.

Since the algebraic Chern character $\ch : K^\alg_*(\cB M) \to \HP_*(\cB M)$ factors through $\KH_*(\cB M)$, any possible difference between $[\varphi(P)]$ and $[\varphi^\prime(P)]$ disappears after applying the Chern character.

\item When we constructed $A^\alg(T)$ we first discussed a class $[P] \in K^\alg_0(\cA M / \cB M)$ which was used in \eqref{eq_defnA_odd}. One can of course apply the boundary map in algebraic $K$-theory to obtain a class
\[
\partial [P] \in K_{-1}^\alg(\cB M)
\]
which is independent of any choices of polynomials satisfying (i)--(iii).

The relation of this class to $A^\alg(T) \in K_1^\alg(\cB M)$ should be as follows: Using results of Corti\~{n}as--Thom \cite{CT} one can prove that $\KH_*(\cB M)$ satisfies Bott periodicity \cite{coarse_KH}; especially we get $\KH_{-1}(\cB M) \cong \KH_1(\cB M)$ via multiplication with the Bott element. 
We believe that $\partial [P]$ and $A^\alg(T)$ coincide after passing to $\KH$-theory and applying this isomorphism but decided to not carry out this argument here since we solely work with $A^\alg(T)$ in our diagrams.

Note that Roe \cite{roe_coarse_cohomology} and Moscovici--Wu \cite{moscovici_wu} use this approach in their works, i.e. work with classes in $K_{-1}^\alg(\cB M)$; though they use not the ``official'' $\smash{K_{-1}^\alg}$ but a variant of it suited for their needs.\qedhere
\end{enumerate}
\end{rem}

\subsection{Comparison with the Connes--Chern character}
\label{SectionConnesMoscovici}

Here, $M$ has to be a smooth manifold. In this section, we discuss commutativity of the diagram
\begin{equation}\label{eq_moscovici_wu}
\begin{tikzcd}
& & K^{\mathrm{an}, \mathrm{lf}}_*(M)
		\ar[ddd, "\substack{\text{Connes--} \\ \text{Chern} \\ \text{character}}"]
\\
	K^{\mathrm{alg}}_*(\mathcal{B}M)
	\ar[d, "\ch"']
	&
		\left\{\substack{\text{1-summable}\\ 
	\text{even/odd}\\\text{Fredholm} \\  \text{modules}}\right\}
	\ar[ru, "{[-]}"]
	\ar[l, "A^{\mathrm{alg}}"']
	\ar[ddr, "{\text{Moscovici--Wu}}"', dashed, bend left=10]
	&
\\
	\HP_*(\mathcal{B}M)
	\ar[d, "\chi"']
\\
	\HX^{\mathrm{per}}_*(M)
	&
	&
	H_*^{\mathrm{dR}, \mathrm{lf}, \mathrm{per}}(M).
			\ar[ll, "c"']
\end{tikzcd}
\end{equation}
which is (without the dashed arrow that we will explain below) the lower left part of \eqref{BigDiagram}. 
Commutativity of this diagram is essentially proved by Moscovici--Wu \cite{moscovici_wu}, but unfortunately, they use a rather different language to state their results.
Below we discuss how to nevertheless extract the desired statement from their paper.

They start with a finitely summable (i.e., $p$-summable for some $p\in [1,\infty)$) Fredholm module and first apply the truncation procedure from \eqref{eq_truncation} with a sequence of covers $\smash{(U_i^{(k)})_{i \in I_k}^{k \in \IN}}$ and corresponding sequence of subordinate partitions of unity satisfying that the diameters of the covers converge (uniformly in $i$) to $0$ as $k \to \infty$. 
This results in a sequence of operators $(\tilde{T}^{(k)})_{k \in \IN}$ with the property that the propagation of these operators converges to $0$ as $k \to \infty$. From these operators they then construct cycles in homology akin to the ones we get when applying the composition $\chi \circ \ch \circ A^\alg$; and since the propagation of the operators goes to $0$, they can then interpret the whole sequence of these homological cycles as a single class in Alexander--Spanier homology of $M$. Identifying the latter with de Rham homology, we get the dashed arrow in the above diagram. Basically by construction, the left part of \eqref{eq_moscovici_wu} commutes, and that the right part commutes is proven by Moscovici--Wu.

Notice that the construction of Moscovici--Wu is essentially a precursor to the localization algebra idea of Guoliang Yu \cite{yu_localization}.

\printbibliography[heading=bibintoc]

\end{document}